\documentclass[12pt]{amsart}
\usepackage[utf8]{inputenc}
\usepackage{color}
\usepackage{float}
\usepackage{times}
\usepackage[hidelinks]{hyperref}
\usepackage{enumerate,latexsym}
\usepackage{latexsym}
\usepackage{subcaption}

\usepackage{fullpage}

\usepackage{amsmath,amsthm,amsfonts,amssymb}
\usepackage{graphicx}
\usepackage{dsfont}

\usepackage{tikz}
\usetikzlibrary{arrows.meta}
\usetikzlibrary{knots}
\usetikzlibrary{hobby}
\usetikzlibrary{arrows,decorations.markings}
\usetikzlibrary{external}
\usetikzlibrary{fadings}

\makeatletter
\def\namedlabel#1#2{\begingroup
 #2%
 \def\@currentlabel{#2}%
 \phantomsection\label{#1}\endgroup
}
\makeatother

\usepackage[lite]{amsrefs}

\renewcommand{\PrintDOI}[1]{\href{http://dx.doi.org/\detokenize{#1}}{doi: \detokenize{#1}}%
	\IfEmptyBibField{pages}{, (to appear in print)}{}}
\theoremstyle{plain}
\newtheorem*{theorem*}{Theorem}
\newtheorem*{thmex*}{Theorem~\ref{example}}
\newtheorem*{thmasymp*}{Theorem~\ref{thmAsymp}}
\newtheorem{theorem}{Theorem}[section]

\newtheorem{lemma}[theorem]{Lemma}

\newtheorem{proposition}[theorem]{Proposition}

\newtheorem{example}[theorem]{Example}

\theoremstyle{definition}
\newtheorem{definition}[theorem]{Definition}

\newcommand{\ben}{\begin{enumerate}}
\newcommand{\een}{\end{enumerate}}

\newcommand{\ed}{\end{document}}

\definecolor{rrr}{rgb}{.9,0,.1}

\definecolor{rr}{rgb}{.8,0,.3}

\usepackage{pdfsync}
\graphicspath{ {./images/} }
\date{}
\title{G-Families of Singquandles}

\title{A G-Family of Singquandles and Invariants of Dichromatic Singular links}

\author[M.I.Sheikh]{Mohd Ibrahim Sheikh}

\address{Department of Mathematics, Graduate School of Natural Sciences
	Pusan National University, Busan 46241, Republic of Korea}
 \email{ibrahimsheikh@pusan.ac.kr}
 \author[M. Elhamdadi]{Mohamed Elhamdadi}
\address{University of South Florida, Tampa, Florida, USA}
\email{emohamed@usf.edu}
\author[D. Ali]{DANISH ALI}
\address{Department of Mathematics, Dalian University of Technology, China}
\email{danishali@mail.dlut.edu.cn}
\date{}
\begin{document}
\maketitle

\begin{abstract}
We introduce and investigate {\it dichromatic singular links}. We also construct $\mathnormal{G}$-Family of singquandles and use them to define counting invariants for unoriented dichromatic singular links. We provide some examples to show that these invariants distinguish some dichromatic singular links.
\end{abstract}

\tableofcontents

\section{Introduction}
A knot is a simple closed curve in three dimensional space $\mathbb{R}^{3}$ and a disjoint union of two or more knots forms a link with two or more components \cites{EN}. Knots and links are categorised in many ways. One way is to use the crossing type as a tool to define a knot or link type. Classical, virtual and singular knots and links serve as examples as they are all recognised by the type of crossing they contain. The other way to define link types is by labelling the components of a classical link. Dichromatic links are defined by using this technique as their components are either labelled by $``1"$ or $``2"$ \cite{HK, HP, K, B, BS}. A singular link is a link with at least one singular crossing. In this paper we use such labelling technique for singular links and define a new type of links which we call {\it dichromatic singular links}.
\par A quandle is an algebraic structure satisfying some axioms that result from the Reidemeister moves for oriented classical knots and links.  If furthermore all right multiplications by fixed elements of the quandle are involutions then such structures are called involutory quandles or Kei's   They are used to investigate unoriented knots and links.  Quandles were independently introduced by Joyce and Matveev \cite{J, M}.  Since then they have been used to construct invariants of knots and links \cite{CCE1, CCEH1, NOS}.
Quandles have been also used to define new algebraic systems by taking a family of quandles at a time. Such systems are called $\mathnormal{G}$-Family of quandles and this notion was introduced in $2013$ by Ishii, Iwakiri, Jang and Oshiro \cite{IMJO}.  A $\mathnormal{G}$-Family of quandles were used to define invariants for handlebody-knots. Also in \cite{LS} Lee and Sheikh  used $\mathbb{Z}_{2}$-Family of quandles to construct algebraic invariants for oriented dichromatic links.
\par In this paper, we introduce the notions of $\mathnormal{G}$-Family of singquandles and dichromatic singular links. A dichromatic singular link is an $\emph{n}$ component singular link with each of its  component labelled as $``1"$ or $``2"$. A singquandle is an algebraic system whose axioms are motivated by Reidemeister moves of unoriented singular knots. By taking a family of such algebaraic systems (Singquandles), we define a new algebraic system which we call $\mathnormal{G}$-Family of singquandles or disingquandle.  The axioms of the latter are motivated by generalized Reidemeister moves of unoriented dichromatic singular links. We discuss various examples and some properties of $\mathnormal{G}$-Family of singquandles, and also show that a $\mathnormal{G}$-Family of singquandles $\mathnormal{X}$ enables us to distinguish unoriented dichromatic singular links by computing their sets of all $\mathnormal{X}$-colorings and proving that these sets are different when their arcs are colored by the elements of $\mathnormal{X}$.
\par This paper is organized as follows.  Section~\ref{SSDL} reviews some preliminaries about singular links, singquandles as well as dichromatic links and their generalized Reidemeister moves. In Section~\ref{SDL} we introduce the notion of dichromatic singular links with some typical examples of unoriented dichromatic singular links. Section~\ref{GFS} introduces the notion of $\mathnormal{G}$-Family of singquandles (disingquandles) with some typical examples of $\mathnormal{G}$-Family of singquandles. Section~\ref{CIUSDL} discusses how $\mathnormal{G}$-Family of singquandles is related to unoriented dichromatic singular links and develop computable invariants for unoriented dichromatic singular links. We discuss some examples which show how the invariants distinguish unoriented dichromatic singular links, and especially how they detect the change of component labelings.
\section{Singular links, Singquandles and Dichromatic Links}\label{SSDL}
In this section we review some preliminaries about singular links, singquandles and dichromatic links. Most of the terminologies of this section can be found in \cite{LS, HN, CEHN}. We begin with the definition of a singular link.
\begin{definition}\label{Def2.1}
A singular link in $\mathbb{S}^3$ is the image of a smooth immersion of $n$ circles in $\mathbb{S}^3$ that has finitely many double points, called singular points.
\end{definition}
A singular link in $\mathbb{R}^{3}$ is represented by a {\it singular link diagram} in the plane $\mathbb{R}^{2}$, which is a classical link diagram with one or more singularities. A singularity is a rigid vertex %or a place 
where a link is glued to itself. Figure~\ref{SingLinks} gives two examples of singular links.
%\vspace{1cm}
\begin{figure}[h]
		%\vspace{0.5cm}
		\centering
		\includegraphics[width=0.6\textwidth]{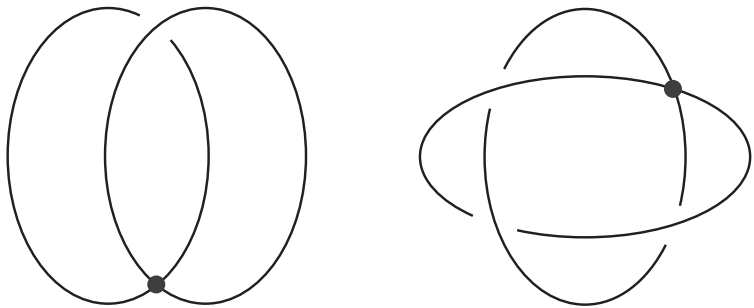}
	\caption{Singular Links}
	\label{SingLinks}
\end{figure}

Two singular links $\mathnormal{L_1}$ and $\mathnormal{L_2}$ are isotopy equivalent if one can be obtained from the other by a finite sequence 
%of classical and 
generalized Reidemeister moves for singular links as shown in the following figure. Let $\mathnormal{D_1}$ and $\mathnormal{D_2}$ be two singular link diagrams in $\mathbb{R}^{2}$ representing $\mathnormal{L_1}$ and $\mathnormal{L_2}$, respectively. Then $\mathnormal{L_1}$ and $\mathnormal{L_2}$ are equivalent if and only if $\mathnormal{D_1}$ and $\mathnormal{D_2}$ can be transformed into each other by a finite sequence of classical and singular Reidemeister moves shown in Figure~\ref{SingRMoves}. 
\begin{figure}[h]
\centering
\includegraphics[width=0.8\textwidth]{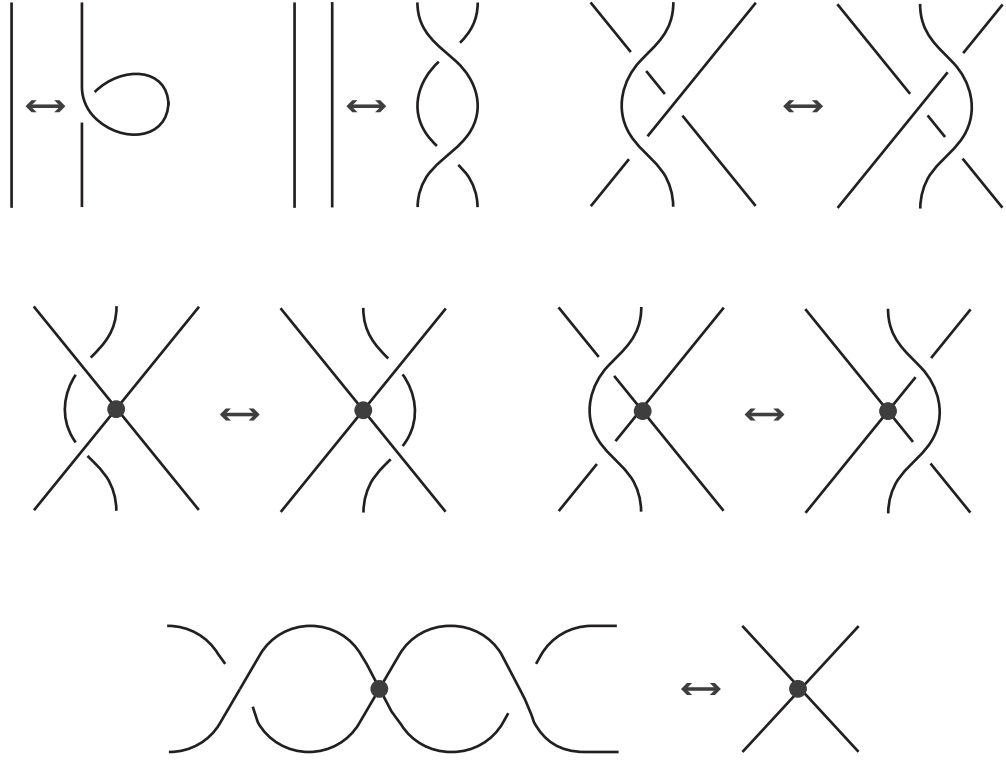}
\caption{Classical and Singular Reidemeister Moves}
\label{SingRMoves}
\end{figure}\\
\begin{definition}\cite{CEHN} \label{Def2.2}
Let $(\mathnormal{X}, *)$ be an involutive quandle. Let $\mathbf{R_1}$ and $\mathbf{R_2}$ be two maps from $\mathnormal{X} \times \mathnormal{X}$ to $\mathnormal{X}$. The quadruple $(\mathnormal{X}, *, \mathbf{R_1}, \mathbf{R_2})$ is called a singquandle if the following axioms are satisfied
%\begin{enumerate}
%    \item $x = \mathbf{R_1}(y, \mathbf{R_2}(x, y)) = \mathbf{R_2}(\mathbf{R_2}(x, y), \mathbf{R_1}(x, y))$,
%    \item $y = \mathbf{R_2}(\mathbf{R_2}(x, y), x) = \mathbf{R_1}(\mathbf{R_2}(x, y), \mathbf{R_1}(x, y))$,
%    \item $\mathbf{R}(x, y) = (\mathbf{R_1}(y, \mathbf{R_2}(x, y)), \mathbf{R_2}(\mathbf{R_2}(x, y), x))$
    
%\end{enumerate}
\[ 
x = \mathbf{R_1}(y, \mathbf{R_2}(x, y)) = \mathbf{R_2}(\mathbf{R_2}(x, y), \mathbf{R_1}(x, y)), \tag{2.2.1} \label{eq:2.2.1} \]
\[
y = \mathbf{R_2}(\mathbf{R_2}(x, y), x) = \mathbf{R_1}(\mathbf{R_2}(x, y), \mathbf{R_1}(x, y)), \tag{2.2.2} \label{eq:2.2.2}
\]
\[
\mathbf{R}(x, y) = (\mathbf{R_1}(y, \mathbf{R_2}(x, y)), \mathbf{R_2}(\mathbf{R_2}(x, y), x)), \tag{2.2.3} \label{eq:2.2.3}
\]
\[
(y * z) * \mathbf{R_2}(x, z) = (y * x) * \mathbf{R_1}(x, z), \tag{2.2.4} \label{eq:2.2.4} 
\]
\[
\mathbf{R_1}(x, y) = \mathbf{R_2}(y * x, x), \tag{2.2.5} \label{eq:2.2.5} 
\]
\[
\mathbf{R_2}(x, y) = \mathbf{R_1}(y * x, x) * \mathbf{R_2}(y * x, x), \tag{2.2.6} \label{eq:2.2.6} 
\]
\[
\mathbf{R_1}(x * y, z) * y = \mathbf{R_1}(x, z * y), \tag{2.2.7} \label{eq:2.2.7} 
\]
\[
\mathbf{R_2}(x * y, z) = \mathbf{R_2}(x, z * y) * y. \tag{2.2.8} \label{eq:2.2.8} 
\]
\end{definition}
We remind the reader that the singquandle axioms come from the generalized Reidemeister moves for unoriented singular knots. Singquandles were introduced as a ramification of quandles 
%defined for classical knots and links and singquandles in turn has also been ramified by different topologists and algebraists in different perspectives.
with the purpose of studying singular links, see for example \cite{CCE1, NOS, CEHN}.
\par The following are few typical examples of singquandles.
\begin{itemize}
    \item For an involutive quandle $(X,*)$ with $x*y= 2y-x$ and $X=\mathbb{Z}_n$, the quadruple $(\mathnormal{X}, *, \mathbf{R_1}, \mathbf{R_2})$ forms a singquandle if and only if the following conditions are satisfied:
    \begin{enumerate}
        \item $\mathbf{R_2}(x, y) = \mathbf{R_1}(x, y) + y - x$,
        \item $\mathbf{R_1}(x, y) = \mathbf{R_1}(2x-y, x) + y - x$,
        \item $\mathbf{R_1}(x, 2y-z) = 2y - \mathbf{R_1}(2y-x, z)$,
        \item $\mathbf{R_2}(2y-x, z) = 2y - \mathbf{R_2}(x, 2x-z)$.
    \end{enumerate}
    \item For an involutive quandle $(X,*)$ where $X$ is a group $G$ and $x*y= yx^{-1}y$, the quadruple $(\mathnormal{X}, *, \mathbf{R_1}, \mathbf{R_2})$ forms a singquandle if and only if the following conditions are satisfied:
    \begin{enumerate}
        \item $\mathbf{R_2}(x, z)z^{-1} yz^{-1} \mathbf{R_2}(x, z) = \mathbf{R_1}(x, z)x^{-1} yx^{-1} \mathbf{R_1}(x, z)$,
        \item $\mathbf{R_1}(x, y) = \mathbf{R_2}(xy^{-1}x, x)$,
        \item $\mathbf{R_2}(x, y) = \mathbf{R_2}(xy^{-1}x, x)[\mathbf{R_1}(xy^{-1}x, x)]^{-1} \mathbf{R_2}(xy^{-1}x, x)$,
        \item $y[\mathbf{R_1}(yx^{-1}y, z)]^{-1}y = \mathbf{R_1}(x, yz^{-1}y)$,
        \item $\mathbf{R_1}(yx^{-1}y, z) = y[\mathbf{R_2}(x, yz^{-1}y)]^{-1}y$.
    \end{enumerate}
\end{itemize}
\begin{definition}\label{Def2.3}
For a positive integer $n \geq 1$.  A dichromatic link is a smooth imbedding of $n$ circles in $\mathbb{R}^{3}$ %that has two 
such that each component is labeled as $``1"$ or $``2"$.
\end{definition}
In $\mathbb{R}^{2}$ every dichromatic link is represented by a dichromatic link diagram which is a classical link diagram with each component labelled either $``1"$ or $``2"$. For example, see Figure~\ref{DcLinks}.  
\begin{figure}[h]
		\vspace{0.5cm}
		\includegraphics[keepaspectratio,width= 10cm,
		height=10cm]{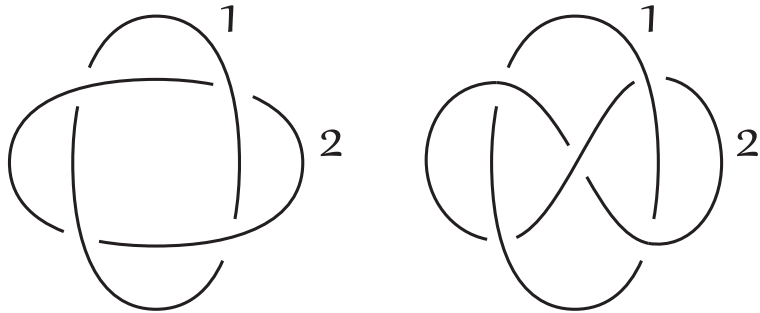}
	\caption{Dichromatic Links}
	\label{DcLinks}
\end{figure}

Two dichromatic links $\mathnormal{L_1}$ and $\mathnormal{L_2}$ are isotopy equivalent if one can be obtained from the other by a finite sequence of generalized
Reidemeister moves for the dichromatic links as shown in the figure \ref{DcMoves}. Let $\mathnormal{D_1}$ and $\mathnormal{D_2}$ be two dichromatic link diagrams in $\mathbb{R}^{2}$ representing $\mathnormal{L_1}$ and $\mathnormal{L_2}$, respectively. Then $\mathnormal{L_1}$ and $\mathnormal{L_2}$ are equivalent if and only $\mathnormal{D_1}$ and $\mathnormal{D_2}$ can be transformed into each other by a finite sequence of generalized Reidemeister moves shown in the following Figure~\ref{DcMoves}.
\begin{figure}
\centering
\includegraphics[width=0.8\textwidth]{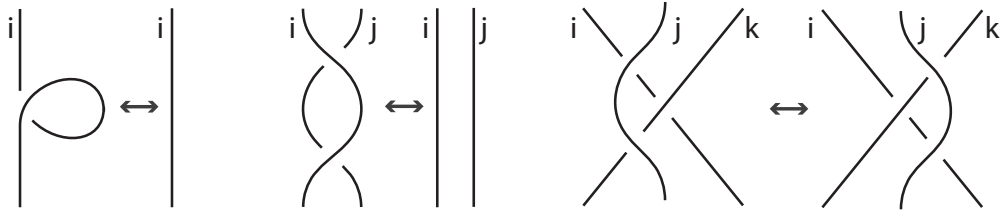}
	\caption{Generalized Reidemeister Moves for Dichromatic Links}
	\label{DcMoves}
\end{figure}

\section{Dichromatic Singular Links}\label{SDL}
This section is devoted to 
%a generalization of singular links to dichromatic links or vice versa. 
dichromatic singular links which is a generalization of singular links.
%The idea is that we add an extra information to a singular link by labelling its components either by $``1"$ or $``2"$. In other words we just add singularities to a dichromatic link. The new type of links obtained are called as {\it Singular Dichromatic Links}. More formally, we introduce the following definition.
To generate a dichromatic singular link we label a singular link's components with $``1"$ or $``2"$.  Thus We have the following definition.
\begin{definition}\label{Def3.1}
A singular link $\mathnormal{L}$ in $\mathbb{R}^{3}$ whose each component is colored (labelled) by either $``1"$ or $``2"$ is called a {\it dichromatic singular link}.
\end{definition}
A dichromatic singular link $\mathnormal{L}$ in $\mathbb{R}^{3}$ is represented by a dichromatic singular link diagram $\mathnormal{D}$ in $\mathbb{R}^{2}$ 
%which is a singular link diagram that has the label $``1"$ or $``2"$ for each component.
in which each component is labelled $``1"$ or $``2"$. 
Figure \ref{SDLinks} shows two examples of unoriented dichromatic singular link diagrams.
\begin{figure}[h]
		%\vspace{0.8cm}
		\includegraphics[keepaspectratio,width= 10cm,
		height=10cm]{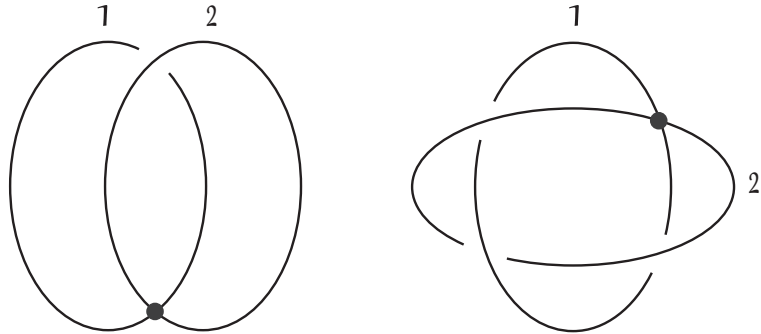}
	\caption{Dichromatic Singular Links}
	\label{SDLinks}
\end{figure}
\par Two dichromatic singular links $\mathnormal{L_1}$ and $\mathnormal{L_2}$ in $\mathbb{R}^{3}$ are {\it ambient isotopic} if there exists a self homeomorphism $h: \mathbb{R}^{3} \rightarrow \mathbb{R}^{3}$ that takes one link to the other and preserves the singularities as well as the labels $``1", ``2"$ such that $h(\mathnormal{L_1})= \mathnormal{L_2}$. Thus two singular dichromatic links $\mathnormal{L_1}$ and $\mathnormal{L_2}$ are equivalent if one can be obtained from the other by a finite sequence of generalized dichromatic singular Reidemeister moves 
%for the singular dichromatic links which 
preserving the label of each component as shown in the Figure \ref{SDRMoves}. Let $\mathnormal{D_1}$ and $\mathnormal{D_2}$ be two dichromatic singular link diagrams in $\mathbb{R}^{2}$ representing $\mathnormal{L_1}$ and $\mathnormal{L_2}$, respectively. Then $\mathnormal{L_1}$ and $\mathnormal{L_2}$ are equivalent if and only if $\mathnormal{D_1}$ and $\mathnormal{D_2}$ can be transformed into each other by a finite sequence of generalized dichromatic singular Reidemeister moves shown in the following Figure \ref{SDRMoves} where ${i, j, k} \in \{1, 2\}$.\\
\begin{figure}[htbp]
\centering
\includegraphics[width=0.8\textwidth]{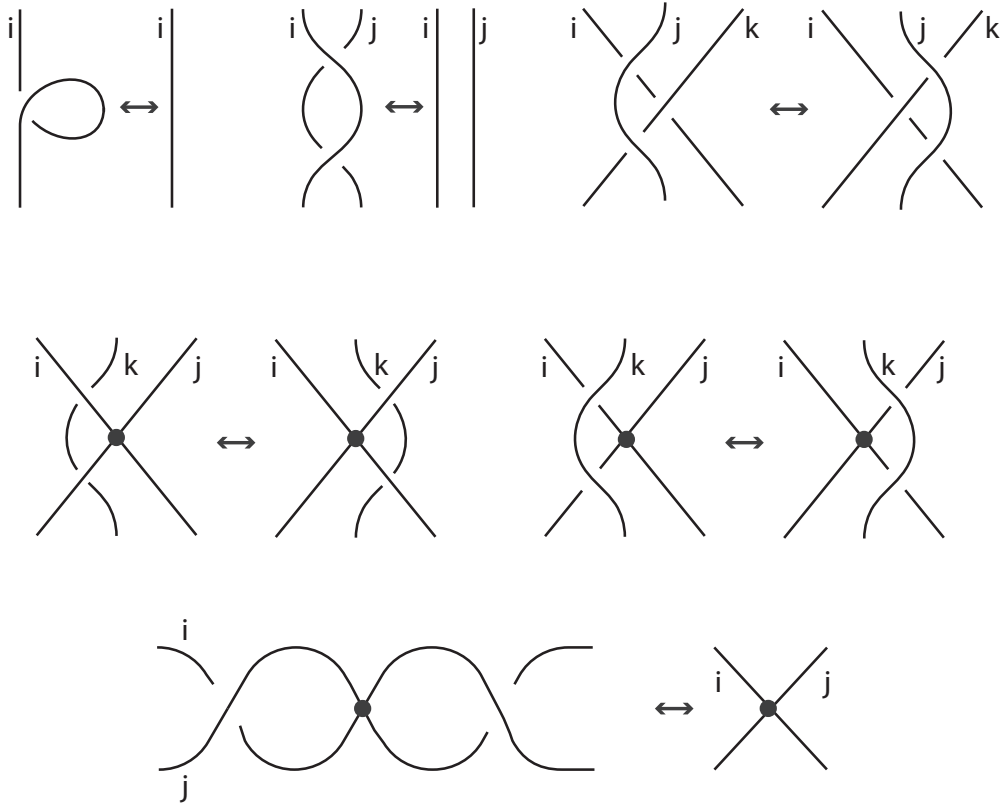}
	\caption{Regular Dichromatic Reidemeister Moves $RI, RII$ and $RIII$ on the top and Dichromatic Singular Reidemeister Moves $RIVa, RIVb$ and $RV$ in the middle and on the bottom.}
	\label{SDRMoves}
\end{figure}
\vspace{0.6cm}
%{\color{red} MAY BE REMOVE THIS:
%\par An $n$-component singular dichromatic link $\mathnormal{L}$ in $\mathbb{R}^{3}$ is called an {\it oriented singular %dichromatic link} if its each component is oriented. An oriented singular dichromatic link in $\mathbb{R}^{3}$ is represented by %an oriented singular dichromatic link diagram $\mathnormal{D}$ in $\mathbb{R}^{2}$.
%}
\par A dichromatic singular link with $n$ components is called as an $n$-component dichromatic singular link. Thus an $n$-component dichromatic singular link in $\mathbb R^3$ can be defined as $L=K_1\cup \cdots \cup K_n$. Taking $n=2$, we obtain $2$-component dichromatic singular links. Some $2$-component unoriented dichromatic singular link diagrams (see p 814 of \cite{Oyamaguchi}) are shown in Figure \ref{TOSDLinks}.
%\begin{figure}
%		%\vspace{0.5cm}
%		\includegraphics[width=0.7\textwidth]{TOSDLinks.eps}
%	\caption{Table of Unoriented Dichromatic Singular Links}
%	\label{TOSDLinks}
%\end{figure}

\begin{proposition}\label{Prop3.1}
Let $\mathnormal{L_1}$ and $\mathnormal{L_2}$ be two unoriented dichromatic singular links in $\mathbb{R}^{3}$ and let $\mathnormal{D_1}$ and $\mathnormal{D_2}$ be two unoriented dichromatic singular link diagrams in $\mathbb{R}^{2}$ representing $\mathnormal{L_1}$ and $\mathnormal{L_2}$, respectively. Then $\mathnormal{L_1}$ and $\mathnormal{L_2}$ are equivalent if and only if $\mathnormal{D_1}$ and $\mathnormal{D_2}$ are transformed into each other by a finite sequence of generalized Reidemeister moves for unoriented dichromatic singular links which preserve the singularities and the label of each component as shown in the Fig. \ref{SDRMoves} where ${i, j, k} \in \{1, 2\}$ and ambient isotopies of $\mathbb{R}^{2}$.
\end{proposition}

%\newpage
\section{$\mathnormal{G}$-Family of Singquandles (Disingquandles)}\label{GFS}
 Before introducing the notion of $\mathnormal{G}$-Family of Singquandles, we first recall the definition of $G$-family of quandles from \cite{IMJO}.
    \begin{definition}
    Given a group $G$ and a set $X$, a $G$-family of quandles, denoted by $(G,X)$, is a choice of quandle operation $*^g$ on the set $X$ for each element $g \in G$ such that the following axioms are satisfied
    \begin{enumerate}
        \item 
        For all $g \in G$ and for all $x \in X$, $x*^gx=x$,
        \item
        For all $g,h \in G$ and for all $x,y \in X$, $(x*^g y)*^h y=x*^{gh}y$,
        \item
        For all $x,y \in X$, $x*^e x=x$, where $e$ is the identity element of $G$,
        \item
        
         For all $x,y,z \in X$, $(x*^g y)*^h z= (x*^h z) *^{h^{-1}gh} (y*^h z)$
        
    \end{enumerate}
    \end{definition}
    The following are two examples of $G$-families of quandles.
    \begin{itemize}
        \item 
        For any group $G$ and any set $X$, defining $x*^g y=x$ for all $x,y \in X$ and all $g \in G$.  This gives a $G$-family of quandles called the \emph{trivial} $G$-family of quandles.
        
        \item
        Let $(X,*)$ be a quandle of \emph{cyclic type} \cite{Tamaru} with cardinality $n$.  Let $R_x$ denotes the right multiplication by $x$, thus by definition ${R_x}^{(n-1)}$ is the identity map.  Then define $x*^iy={R_y}^i(x)$ then it is shown in Proposition 2.3 of \cite{IMJO} that $(\mathbb{Z},X)$ is a $\mathbb{Z}$-family of quandles and also $\mathbb{Z}_{(n-1)}$-family of quandles.
    \end{itemize}
\par A $G$-family of quandles $(G,X)$ induces a quandle operation on the set $G \times X$ by
    \[
    (g,x)*(h,y)=(h^{-1}gx,x*^{h}y).
    \]
\par The notion of $G$-family of quandles was introduced by Ishii, Iwakiri, Jang and Oshiro in 2013 in \cite{IMJO} in order to produce invariants of handlebody knots.  They defined coloring invariants and cocycle invariants of handlebody knots.  They used these invariants to detect chirality of some handlebody knots. Later in $2015$, Ishii independently studied the notion of $G$-family of quandles in connection with the multiple conjugation quandle and showed that the later one can be obtained from the first one. In $2017$ and $2018$ Ishii, Nelson and Ishii, Iwakiri, Kamada, Kim, Matsuzaki, Oshiro respectively, %took forward 
used this work and introduced the notions of partially  multiplicative  biquandles and multiple conjugation biquandle. In $2021$ Lee and Sheikh jointly used $G$-family of quandles to construct algebraic invariants for oriented dichromatic links \cite{LS}. We introduce the following definition.
\begin{definition}\label{Def4.2}
Let $X$ be a set equipped with two binary operations $*_1$ and $*_2$ such that both $(X,*_1), (X,*_2)$ are involutive quandles.  Let $\mathbf{R_1}, \mathbf{R_2}$ be two maps from $\mathnormal{X} \times \mathnormal{X}$ to $\mathnormal{X}$ such that the quadruples $(\mathnormal{X}, *_1, \mathbf{R_1}, \mathbf{R_2})$ and $(\mathnormal{X}, *_2, \mathbf{R_1}, \mathbf{R_2})$ are singquandles. Then the quintuple $(\mathnormal{X}, *_1, *_2, \mathbf{R_1}, \mathbf{R_2})$ is called a {\it disingquandle} or $\mathbb{Z}_{2}$-{\it family of singquandles} if the following axioms are satisfied
\[
(x *_{1} y) *_{2} z =  (x *_{2} z) *_{1} (y *_{2} z), \tag{4.2.1} \label{eq:4.2.1}
\]
\[
(x *_{2} y) *_{1} z =  (x *_{1} z) *_{2} (y *_{1} z), \tag{4.2.2} \label{eq:4.2.2}
\]
\[
(y *_1 z) *_2 \mathbf{R_2}(x, z) = (y *_2 x) *_1 \mathbf{R_1}(x, z), \tag{4.2.3} \label{eq:4.2.3} 
\]
\[
(y *_2 z) *_1 \mathbf{R_2}(x, z) = (y *_1 x) *_2 \mathbf{R_1}(x, z), \tag{4.2.4} \label{eq:4.2.4} 
\]
\[
\mathbf{R_2}(x, y) = \mathbf{R_1}(y *_1 x, x) *_2 \mathbf{R_2}(y *_1 x, x), \tag{4.2.5} \label{eq:4.2.5} 
\]
\[
\mathbf{R_2}(x, y) = \mathbf{R_1}(y *_2 x, x) *_1 \mathbf{R_2}(y *_2 x, x), \tag{4.2.6} \label{eq:4.2.6} 
\]
\end{definition}
The above axioms of a disingquandle come from the generalized dichromatic singular Reidemeister moves shown in Figure~\ref{SDRMoves} when we take the coloring rule shown in Figure~\ref{SDLinkColoring} under consideration.
\begin{figure}[H]
		%\vspace{0.5cm}
		\includegraphics[width=0.8\textwidth]{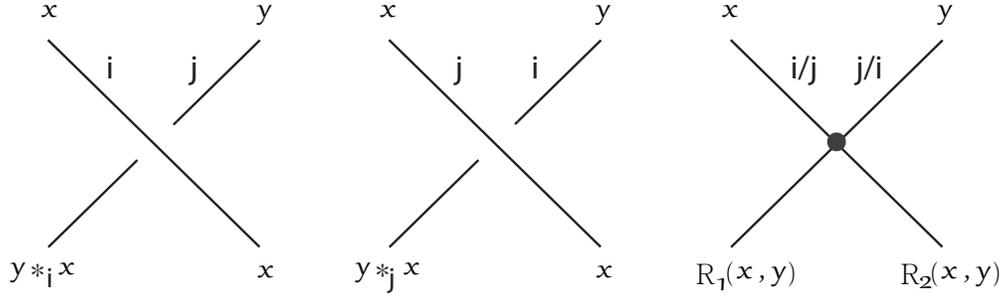}
	\caption{Coloring by a disingquandle}
	\label{SDLinkColoring}
\end{figure}
%	\begin{figure}[ht]
%		\centerline{
%			\xy (-10,10);(-2,2) **@{-},(10,-10);(2,-2)**@{-} **@{-}, 
%			(10,10);(-10,-10) **@{-} **@{-},(12,-13) *{x\rhd_j y},
%			(-13,-12) *{y},(-13,12) *{x}, (13,12) *{y},(-8.5,4) *{i},(7.5,3.5) *{j},
%			\endxy\qquad\qquad\qquad
%			\xy (10,-10);(-10,10) **@{-} **@{-},(10,10);(2,2) **@{-},
%			(-2,-2);(-10,-10) **@{-} **@{-},(12,-13) *{y}, (-13,12) *{y},
%			(-12,-13) *{x\rhd_i y},(13,12) *{x},(-8.5,4) *{i},(7.5,3.5) *{j},
%			\endxy}
%		\vskip.1cm
%		\caption{The dikei colorings at crossings with $i, j \in \{1,2\}$}\label{fig8}
%	\end{figure}
%\begin{tikzpicture}[scale=0.4]
%   
%   \draw (0,0) coordinate (A);
%   \draw (0.35,0.35) coordinate (E);
%   \draw (1,1) coordinate (B);
%   \draw (0.65,0.65) coordinate (F);
%   \draw (1,0) coordinate (C);
%   \draw (0.35,0.65) coordinate (G);
%   \draw (0,1) coordinate (D);
%   \draw (0.65,0.35) coordinate (H);
%   \draw (A)--(E);\draw (B)--(F);\draw (C)--(D);   
%   \coordinate (X) at (intersection of A--B and C--D);
%\end{tikzpicture}   

%\begin{tikzpicture}[scale=2.0]   
%   \draw (10,10) coordinate (A);
%   \draw (15,15) coordinate (B);
%   \draw (15,10) coordinate (C);
%   \draw (10,15) coordinate (D);
%   \draw (A)--(B);\draw (C)--(D);  
%   \coordinate (X) at (intersection of A--B and C--D);
%\end{tikzpicture}	
\par The following lemma is motivated by the above construction.
\begin{lemma}\label{lem4.1}
The set of colorings of a dichromatic singular link by a disingquandle does not change by the dichromatic singular Reidemeister moves shown in Figure \ref{SDRMoves}.
\end{lemma}

\begin{proof}
 As in the case of classical and singular knot theories, there is one to one correspondence between colorings before and after each of the dichromatic singular Reidemeister moves.  The invariance follows directly from the equations~\ref{eq:4.2.1}, \ref{eq:4.2.2}, \ref{eq:4.2.3}, \ref{eq:4.2.4}, \ref{eq:4.2.5} and \ref{eq:4.2.6} given in Definition~\ref{Def4.2}. %The corresponding figure is after Figure 9 and before Def 5.3.}
\end{proof}

\begin{example}\label{Ex1}
Let $(\mathnormal{X}, *_1, \mathbf{R_1}, \mathbf{R_2})$ and $(\mathnormal{X}, *_2, \mathbf{R_1}, \mathbf{R_2})$ be two singquandles such that such that $x \;*_1\; y= x = x \;*_2\; y$ and $\mathbf{R_1}(x,y)=\mathbf{R_2}(x,y)$, then $(\mathnormal{X}, *_1, *_2, \mathbf{R_1}, \mathbf{R_2})$ forms a disingquandle.
\end{example}

\begin{example}\label{Ex2}
Let $(\mathnormal{X}, *_1, \mathbf{R_1}, \mathbf{R_2})$ and $(\mathnormal{X}, *_2, \mathbf{R_1}, \mathbf{R_2})$ be two singquandles.  If for all $x,y \in X$ we have $x *_1 y= x *_2 y$ then $(\mathnormal{X}, *_1, *_2, \mathbf{R_1}, \mathbf{R_2})$ forms a disingquandle.
\end{example}
Now Example~\ref{Ex2} combined with Proposition 3.6 in \cite{CEHN} gives the following example. 
\begin{example}\label{Ex3}
Let $\Lambda=\mathbb{Z}[t,B]/(t^2-1, B(1+t), t-(1-B)^2)$ and $X$ be an $\Lambda$-module.  Define $x*_1 y=tx+(1-t)y, R_1(x,y)=(1-t-b)x+(t+b)y$ and $R_2(x,y)=(1-B)x +By$, then by setting $*_2=*_1$, then one obtains that $(\mathnormal{X}, *_1, *_2, \mathbf{R_1}, \mathbf{R_2})$ forms a disingquandle.
\end{example}

%{\color{red}
\begin{example}\label{Ex3}
Let $X$ be a module over $\Lambda=\mathbb{Z}[t]$. Define $x*_1 y= x*_2y=tx+(1-t)y, R_1(x,y)=(1-t-B)x+(t+B)y$ and $R_2(x,y)=(1-B)x +By$. Setting $t=-1$ and $X=\mathbb{Z}_{7}$, then the quintuple $(\mathnormal{X}, *_1, *_2, \mathbf{R_1}, \mathbf{R_2})$ forms a disingquandle if $B=4$ or if $B=5$. 
\end{example}
This example can be generalized to $\mathbb{Z}_p$, where $p$ is a prime as follows.
\begin{example}\label{examplePrimeNumber}
Let $p$ be an odd prime and let $B \in \mathbb{Z}_p$.  Consider $\mathbb{Z}_p$ with $x*_1 y= x*_2y=-x+2y, R_1(x,y)=(2-B)x+(-1+B)y$ and $R_2(x,y)=(1-B)x +By$.  Let $\zeta$ be a primitive root of unity in $\mathbb{Z}_p$ so that $\zeta^{\frac{p-1}{2}}=-1$.  By choosing $1-B=\zeta^{\frac{p-1}{2}}$ we obtain that $(\mathbb{Z}_p, *_1, *_2, \mathbf{R_1}, \mathbf{R_2})$ forms a disingquandle.

\end{example}

\begin{example}\label{Ex4}
Let $X=G$ be a multiplicative group with the involutive quandle operations $x*_1 y=x*_2 y=yx^{-1}y$ (core quandle on $G$), then a direct computation gives the fact that the quintuple $(\mathnormal{X}, *_1, *_2, \mathbf{R_1}, \mathbf{R_2})$ forms a disingquandle if and only if $\mathbf{R_1}$ and $\mathbf{R_2}$ satisfies the following equations:
\[
\mathbf{R_2}(x, z)z^{-1}yz^{-1}\mathbf{R_2}(x, z) = \mathbf{R_1}(x, z)x^{-1}yx^{-1}\mathbf{R_1}(x, z), \tag{5.1} \label{eq:5.1} 
\]
\[
\mathbf{R_2}(x, y) = \mathbf{R_2}(xy^{-1}x, x)[\mathbf{R_1}(xy^{-1}x, x)]^{-1}\mathbf{R_2}(xy^{-1}x, x), \tag{5.2} \label{eq:5.2} 
\]
\end{example}
A straightforward computation gives the following solution
%$$\begin{equation*}

  \[
  \mathbf{R_1}(x, y)=x \;\textit{and}\; \mathbf{R_2}(x, y)=y, \;\textit{for all}\; x, y, z \in G.
    \]
%\end{equation*}$$

%\begin{enumerate}
%    \item $\mathbf{R_1}(x, z)=x$ and $\mathbf{R_2}(x, z)=z$.
%    \item $\mathbf{R_1}(x, y)=x$ and $\mathbf{R_2}(x, y)=y$.{\color{green} but these are the same solutions}
%\end{enumerate}
%{\color{blue} 
Now assume that $G$ is an abelian group without $2$-torsion, so that $x*y=-x+2y$, then $(\mathnormal{X}, *_1, *_2, \mathbf{R_1}, \mathbf{R_2})$ forms a disingquandle if and only if $R_2(x,y)=R_1(x,y)+y-x,$ where $R_1$ satisfies the identity $R_1(x,y)=R_1(-x+2y,x)+y-x$.  For example for any integer $m$, the map $R_1(x,y)=mx+(2m+1)y$ give a solution.  Thus we have a family of solutions parametrized by the integer $m$:
\[
R_1(x,y)=mx+(2m+1)y, \; R_2(x,y)=(m-1)x+2(m+1)y.
\]
%}
\begin{definition} \label{Def4.3}
A map $f:X\rightarrow Y$ is called a homomorphism of disingquandle $(\mathnormal{X}, *_1, *_2, \mathbf{R_1}, \mathbf{R_2})$ and $(\mathnormal{Y}, *'_1, *'_2, \mathbf{R'_1}, \mathbf{R'_2})$ if the following conditions are satisfied for all $x,y,z \in X$
\begin{enumerate}
    \item[(i)] $f(x*_1y)=f(x)*'_1f(y)$, 
    \item[(ii)] $f(x*_2y)=f(x)*'_2f(y)$,
    \item[(iii)] $f(\mathbf{R_1}(x,y))=\mathbf{R'_1}(f(x),f(y))$,
    \item[(iv)] $f(\mathbf{R_2}(x,y))=\mathbf{R'_1}(f(x),f(y))$.
\end{enumerate}
\par If a homomorphism of disingquandle is bijective, then it is called an isomorphism of disingquandle. We say that two $\mathbb{Z}_{2}$-families of singquandles are isomorphic if there exists an ismorphism of disingquandle between them.
\end{definition}
\begin{definition}\label{Def4.4}
Let $(\mathnormal{X}, *_1, *_2, \mathbf{R_1}, \mathbf{R_2})$ be a disingquandle. A subset $Y\subset X$ is called a sub-disingquandle if $(\mathnormal{Y}, *_1, *_2, \mathbf{R_1}, \mathbf{R_2})$ is itself a disingquandle. 
\end{definition}
\begin{example}\label{Ex5}
%{\it IS IT POSSIBLE TO FIND AN EXAMPLE?}
%{\color{blue}
We use Example~\ref{Ex4} to get the following 2 examples:
\begin{itemize}
\item
Let $X=\mathbb{Z}_9$ be the dihedral quandle with $x*y=-x+2y$, $R_1(x,y)=mx+(2m+1)y, \; R_2(x,y)=(m-1)x+2(m+1)y.$. Then $(\mathnormal{Y}, *_1, *_2, \mathbf{R_1}, \mathbf{R_2})$ is itself a  disingquandle with $Y=\mathbb{Z}_3$.

\item
Let $X=\mathbb{Z}_{25}$ be the dihedral quandle with $x*y=-x+2y$, $R_1(x,y)=mx+(2m+1)y, \; R_2(x,y)=(m-1)x+2(m+1)y.$. Then $(\mathnormal{Y}, *_1, *_2, \mathbf{R_1}, \mathbf{R_2})$ is itself a  disingquandle with $Y=\mathbb{Z}_5$.

\end{itemize}
\end{example}
Given a homomorphism of disingquandles, we obtain the following lemma.
\begin{lemma}\label{Lem4.1}
The image {\it Im(f)} of any homomorphism of disingquandle $f$ defined from $(\mathnormal{X}, *_1, *_2, \mathbf{R_1}, \mathbf{R_2})$ to $(\mathnormal{Y}, *'_1, *'_2, \mathbf{R'_1}, \mathbf{R'_2})$ is always a sub-disingquandle. 
\end{lemma}
\begin{proof}
Given that $f:(\mathnormal{X}, *_1, *_2, \mathbf{R_1}, \mathbf{R_2})\rightarrow (\mathnormal{Y}, *'_1, *'_2, \mathbf{R'_1}, \mathbf{R'_2})$ is a homomorphism. Then the equations $(i), (ii), (iii)$ and $(iv)$ of Definition \ref{Def4.3} imply that {\it Im(f)} is closed under $*_1, *_2, \mathbf{R_1}$ and $\mathbf{R_2}$. Then the axioms of disingquandle are satisfied in $Y$. Hence they are automatically satisfied in {\it Im(f)}. This ends the proof of the lemma.
\end{proof}
Now we introduce the notion of {\it fundamental disingquandle} of an unoriented dichromatic singular link and provide an illustrative example. Let $D$ be a diagram of an unoriented dichromatic singular link $L$ in $\mathbb{R}^{2}$.  We define the {\it fundamental disingquandle} of $D$, denoted by $\mathcal{DSQ}(D)$, as the set of equivalence classes of disingquandle words $W$-$\mathcal{DSQ}(D)$ under the equivalence relation generated by the axioms of disingquandle and the crossing relations shown in Figure \ref{SDLinkColoring}, where $W$-$\mathcal{DSQ}(D)$ are defined by taking a set of generators $X = \{x_{1}, x_{2}, x_{3}, ..... , x_{n} \}$ which corresponds bijectively with the semi arcs in $D$, recursively by the following two rules:
\begin{enumerate}
\item $X \subset$ $W$-$\mathcal{DSQ}(D)$,\\
\item If $x, y \in$ $W$-$\mathcal{DSQ}(D)$, then\\ $x*_1y,
x*_2y, \mathbf{R_1}(x,y), \mathbf{R_2}(x,y) \in$ $W$-$\mathcal{DSQ}(D)$.
\end{enumerate}

\begin{example}
    Consider the following unoriented dichromatic singular link $L$.  
    \begin{figure}[H]
		%\vspace{0.5cm}
		\includegraphics[keepaspectratio,width= 10cm,
		height=8cm]{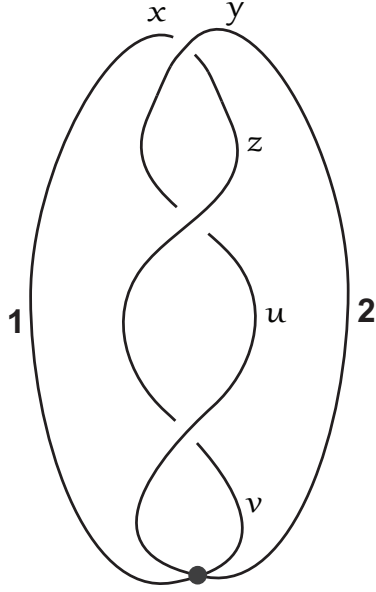}
	\caption{Fundamental Disingquandle of Unoriented Dichromatic Singular Links}
	\label{DSQLinks}
\end{figure}
    The {\it fundamental disingquandle} of $L$ is given by 
    \[
    \mathcal{DSQ}(L)=\langle x,y,z,u,v |\; z=x*_2y; u=y*_1z; v=z*_2u; x=R_1(u,v); y=R_2(u,v) \rangle.  
    \]
    This presentation can be simplified to the following presentation of $\mathcal{DSQ}(L)$
    \[
   \langle x,y|\;x=R_1(y*_1( x*_2y),(x*_2y)*_2 (y*_1( x*_2y)));\\ y=R_2(y*_1( x*_2y),(x*_2y)*_2 (y*_1( x*_2y))) \rangle.  
    \]
    
\end{example}

%}
\section{Computable Invariants for Unoriented Dichromatic Singular Links}\label{CIUSDL}
 Let $D$ be an unoriented dichromatic singular link diagram and let $\mathcal A(D)$ denote the set of all arcs of $D$. Let $(X,*_1, *_2, \mathbf{R_1}, \mathbf{R_2})$ be a disingquandle.  A {\it disingquandle coloring} of $D$ by $X$, or simply {\it disingquandle $X$-coloring} of $D$, is a map $\mathcal{C}: \mathcal A(D) \rightarrow X$ such that at every classical and singular crossing, the relations depicted in Figure \ref{SDLinkColoring} hold. The disingquandle element $\mathcal{C}(s)$ is called a {\it color} of the arc $s$ and the pair $(D, \mathcal C)$ is called the {\it $X$-colored unoriented dichromatic singular link diagram by $\mathcal C$}. The set of all disingquandle $X$-colorings of $D$ is denoted by ${\rm Col}^{dsq}_X(D)$. Then we have the following:
\begin{lemma}\label{lem5.1} 
Let $D$ and $D'$ be two unoriented dichromatic singular link diagramss in $\mathbb R^2$ that can be transformed into each other by unoriented generalized dichromatic singular Reidemeister moves as shown in the Figure \ref{SDRMoves}. Then for any finite disingquandle $X$, there is a one-to-one correspondence between ${\rm Col}^{dsq}_X(D)$ and ${\rm Col}^{dsq}_X(D')$. 
\end{lemma}
\begin{proof}
It suffices to prove the assertion for the case that $D'$ is obtained from $D$ by a single an unoriented generalized dichromatic singular Reidemeister move. Let $E$ be an open disk in $\mathbb R^2$ where the unoriented generalized dichromatic singular Reidemeister move under consideration is applied. Then $D\cap (\mathbb R^2 - E) = D'\cap (\mathbb R^2 -E)$. Now let $\mathcal C$ be a disingquandle $X$-coloring of $D$. Since $(X,*_1, \mathbf{R_1}, \mathbf{R_2})$ and $(X,*_2, \mathbf{R_1}, \mathbf{R_2})$ are both singquandles by the disingquandle definition \ref{Def4.2}, it is obviously seen from the Figure \ref{SDRMoves} that the restriction of $\mathcal C$ to $D\cap (\mathbb R^2 - E) (=D'\cap (\mathbb R^2 -E))$ can be extended to a unique disingquandle $X$-coloring of $D'$ for unoriented generalized dichromatic singular Reidemeister moves $RI, RII$ and $RIII$. Also, using the disingquandle axioms \ref{eq:4.2.1} to \ref{eq:4.2.6}, it is easily seen from the Figure \ref{SDRMoves} that the restriction of $\mathcal C$ to $D\cap (\mathbb R^2 - E) (=D'\cap (\mathbb R^2 -E))$ can be extended to a unique disingquandle $X$-coloring of $D'$ for an unoriented generalized dichromatic Reidemeister moves $RIVa, RIVb$ and $RV$. This completes the proof.
\end{proof}
\par In an $X$-colored unoriented dichromatic singular link diagram $(D, \mathcal C)$, we think of elements of a disingquandle $X$ as labels for the arcs in $D$ with different operations at crossings as shown in Figure \ref{SDLinkColoring}. Then it is seen from Lemma \ref{lem5.1} that the disingquandle axioms of Definition \ref{Def4.2} are transcriptions of a generating set of unoriented generalized Reidemeister moves for unoriented dichromatic singular links which are sufficient to generate any other unoriented generalized dichromatic singular Reidemeister moves. That is, the axioms \ref{eq:4.2.1} and \ref{eq:4.2.2} come from the unoriented generalized dichromatic singular Reidemeister move $RIVa$, the axioms \ref{eq:2.2.3} and \ref{eq:2.2.4} come from the unoriented generalized dichromatic singular Reidemeister move $RIVb$ and the axioms \ref{eq:2.2.5} and \ref{eq:2.2.6} come from the unoriented generalized dichromatic singular Reidemeister move $RV$ as seen in Figure \ref{SDRMoves}.
	
\begin{theorem}\label{theorem5.1} 
Let $L$ be an unoriented dichromatic singular link in $\mathbb R^3$ and let $D$ be a diagram of $L$. Then for any finite disingquandle $X$, the cardinality $\sharp{\rm Col}^{dsq}_X(L)$ is an invariant of $L$.
\end{theorem}
	
\begin{proof}
		Let $D'$ be any other unoriented dichromatic singular link diagram of $L$ obtained from $D$ by applying a finite number of unoriented generalized dichromatic singular Reidemeister moves. Then it is direct from Lemma \ref{lem5.1} that $\sharp{\rm Col}^{dsq}_X(D')=\sharp{\rm Col}^{dsq}_X(D)$. This completes the proof.
	\end{proof}	
If $X$ is a finite disingquandle, we call the cardinality $\sharp{\rm Col}^{dsq}_X(D)$ the {\it disingquandle $X$-coloring number} or the {\it disingquandle counting invariant} of $L$, and denote it by $\mathbb Z^{dsq}_X(L)$, i.e., $\mathbb Z^{dsq}_X(L)=\sharp{\rm Col}^{dsq}_X(D).$

%\begin{figure}
%    \centering
%    \includegraphics{Crossing1.eps}
%    \caption{Caption}
%    \label{fig:my_label}
%\end{figure}
\begin{theorem}\label{Th5.2}
Let $L$ be an unoriented dichromatic singular link and let $X$ be a disingquandle. Then there is a one-to-one correspondence between ${\rm Col}^{dsq}_X(L)$ and ${\rm Hom}(\mathcal{DSQ}(L),X)$. Consequently, $\mathbb Z^{dsq}_X(L)=\sharp{\rm Hom}(\mathcal{DSQ}(L),X).$
\end{theorem}
\begin{proof}
Since the disingquandle $X$-colorings of $L$ generate the fundamental disingquandle $\mathcal{DSQ}(L)$ of a link $L$ which is generated by its arc labels. Also each arc of $L$ is assigned an element of $X$, for a disingquandle $X$-coloring of $L$, so we can associate each coloring a map $f:\mathcal{DSQ}(L)\rightarrow X$ where if an arc is labelled $a$ in the fundamental disingquandle and is assigned the color $x\in X$, then $f(a)=x$. This completes the proof.
%$a \mapsto x$ by $f$.
\end{proof}

Now we give an example.
\begin{example}\label{exmp-dsqci}
 Now, we give an explicit example of \emph{three unoriented dichromatic singular links} $L_1$, $L_2$ and $L_3$ and show that the coloring invariant distinguishes them from each other.  Consider the singquandle
 $(\mathnormal{X}, *, \mathbf{R_1}, \mathbf{R_2})$, where $X=\mathbb{Z}_6$, $x*_1y=x*_2y=-x+2y=x*y$, $R_1(x,y)=x+3,$ and $R_2(x,y)=3x^2+3x+y+3$ (see page 9 of \cite{CCEH}).  By checking directly that the equations of Definition~\ref{Def4.2} hold we obtain that the quintuple $(\mathnormal{X}, *_1,*_2, \mathbf{R_1}, \mathbf{R_2})$ form a disingquandle.  Now coloring the two top arcs of link $L_1$ by $x$ and $y$ as in the figure \ref{SDL1} below gives that the coloring equations are:
\[
x=R_1(R_1(x,y),R_2(x,y)) \quad \mbox{and} \quad y=R_2(R_1(x,y),R_2(x,y)).
\]
\begin{figure}[h]
\includegraphics[width=0.35\textwidth]{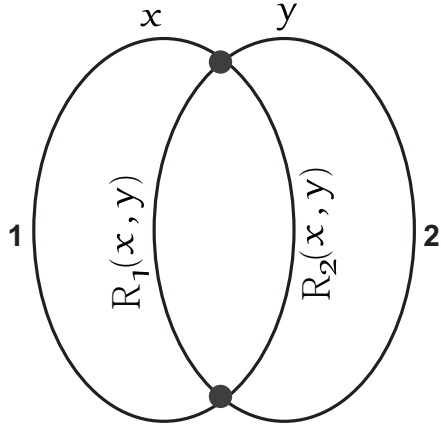}
	\caption{Unoriented Dichromatic Singular Link($L_1$)}
	\label{SDL1}
\end{figure}
One then gets the system,
$\begin{cases}
x=3+3+x,\\
y=3+3(3+x) +3(3+x)^2+(3+3x+3x^2).
\end{cases}$\\
Any pair $(x,y)$ gives a solution to this system over $\mathbb{Z}_6$ and thus the set ${\rm Col}^{dsq}_X(L_1)$ is equal to $\mathbb{Z}_6^2$.\\
Now coloring the link $L_2$ as in the figure \ref{SDL2} below gives that the coloring equations are:
\[
R_1(R_1(x,y),x*R_1(x,y))=R_2(x,y)*y \quad \mbox{and} \quad R_2(R_1(x,y),x*R_1(x,y))=y. 
\]
\begin{figure}[h]
		%\vspace{0.5cm}
		\includegraphics[width=0.45\textwidth]{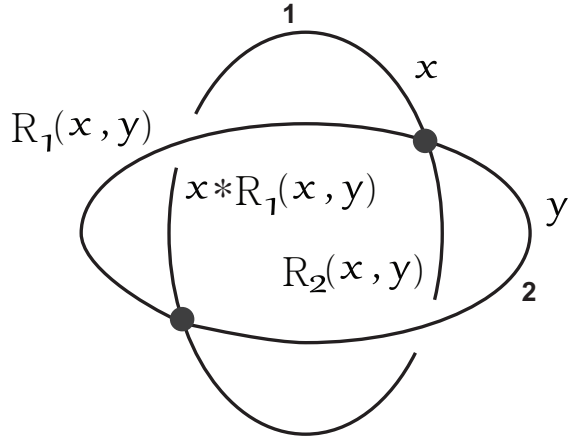}
	\caption{Unoriented Dichromatic Singular Link($L_2$)}
	\label{SDL2}
\end{figure}
One then obtain that the solution is given by $y=3x^2+4x+3$, thus the ${\rm Col}^{dsq}_X(L_2)$ is \[
\{(0,3),(1,4),(2,5),(3,0),(4,1),(5,2) \}.
\]
Now we consider the link $L_3$ (dichromatic singular Whitehead) as in the following figure \ref{SDL3}.  
\begin{figure}[h]
		%\vspace{0.5cm}
		\includegraphics[width=0.49\textwidth]{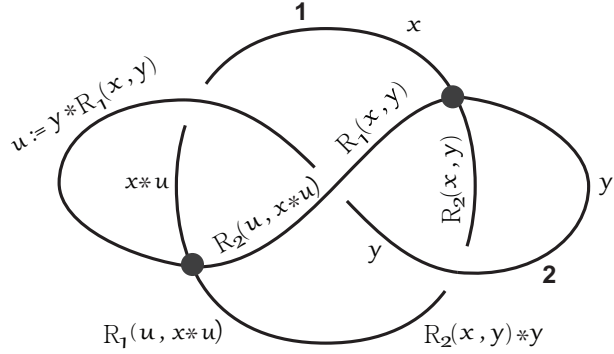}
	\caption{Unoriented Dichromatic Singular Link($L_3$)}
	\label{SDL3}
\end{figure}

The coloring equations are:
\[
R_2(y*R_1(x,y), x* ( y*R_1(x,y)))=R_1(x,y),\]
and
\[R_1(y*R_1(x,y), x* ( y*R_1(x,y)))= R_2(x,y)*y.
\]\\

The system of these two equations reduces to
$
\begin{cases}
0=3y^2+y+2x,\\
0=3x^2+x+2y,
\end{cases}
$\\
and thus we obtain that $2(y-x)=0$ giving $x=y$ or $y=x+3$.\\

Then ${\rm Col}^{dsq}_X(L_3)=\{(x,x), x \in X \} \;\cup \; \{(x,x+3), x \in X \} $. \\ 
\par Thus the \emph{three} links $L_1$, $L_2$ and $L_3$ are \emph{pairwise} distinct.  

\end{example}

\begin{example}
Let $1^2_1, 3^2_1, 4^2_1, 5^2_1, 5^2_2, 5^2_3, 6^2_1, 6^2_2, 6^2_3, 6^2_4, 6^2_5, 6^2_6, 6^2_7, 6^2_8, 6^2_9, 6^2_{10}, 6^2_{11},$ and $6^2_{12}$ be the eighteen unoriented dichromatic  singular links in Figure \ref{TOSDLinks} and let $X$ be the disingquandle in Example \ref{exmp-dsqci}. By similar calculations as in the example, we obtain the following table:

\begin{equation*}
%\begin{aligned}
{\begin{array}{|c|c|}
\hline
L & \# {\rm Col}^{dsq}_X(L) \\\hline
 6^2_2 & 0 \\\hline
  6^2_6 & 2 \\\hline
  4^2_1, 6^2_{12} & 18 \\\hline
1^2_1,3^2_1, 5^2_1, 5^2_2, 5^2_3, 6^2_1, 6^2_3, 6^2_4,  6^2_5,6^2_7, 6^2_8, 6^2_9, 6^2_{10}, 6^2_{11} & 6 \\\hline

\end{array}}
\end{equation*}\\

%\begin{equation*}
%{\begin{array}{|c|c|}\hline
%L & \mathbb Z^{dsq}_X(L) or {\rm Col}^{dsq}_X(L) \\\hline
%1^2_1, 5^2_1, 5^2_2, 5^2_3, 6^2_1, 6^2_5, 6^2_6, 6^2_9, 6^2_{11} & \{(x,x+3); x \in X \} \\\hline
%3^2_1, 6^2_2 & \phi \\\hline
%4^2_1 & \{(0,1), (0,3), (0,5), (3,1), (3,3), (3,5) \} \\\hline
%6^2_3 & \{(0,3), (3,0) \} \\\hline
%6^2_4 & \{(0,3), (1,0), (3,0), (5,0) \} \\\hline
%6^2_7,  6^2_8 & \{(x,x), (x,x+3); x \in X \} \\\hline
%6^2_{10} & \{(x,-x+3), (x,-x-1); x \in X \} \\\hline
%6^2_{12} & \{(x,x+1), (x,x+3), (x,x+5); x \in X \} \\\hline
%\end{array}}
%\end{equation*}

This table shows that the disingquandle counting invariant $\mathbb Z^{dsq}_X(L)$ distinguishes some of these eighteen unoriented dichromatic singular links.
\begin{figure}[h]
		%\vspace{0.5cm}
		\includegraphics[width=0.75\textwidth]{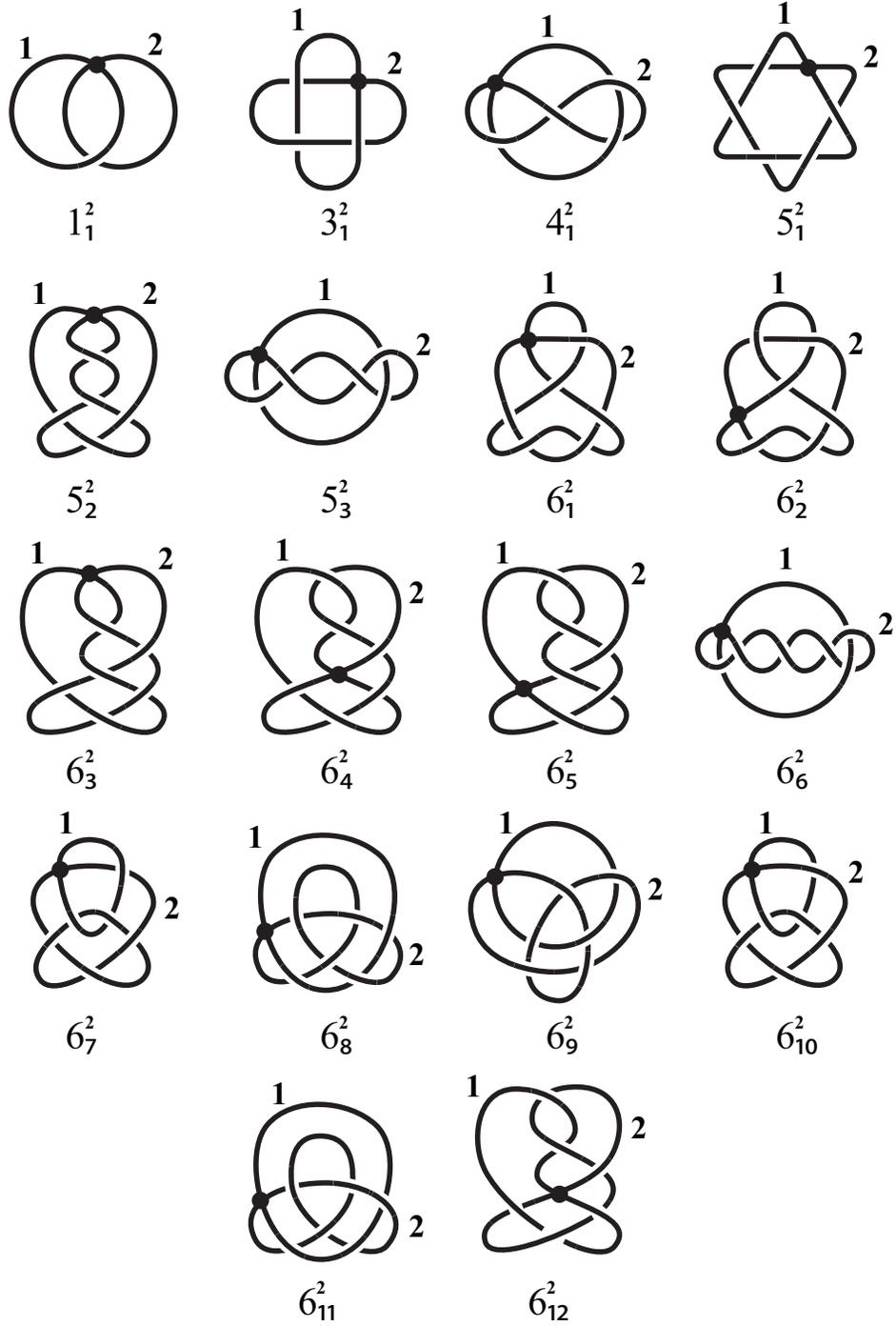}
	\caption{Table of Unoriented Dichromatic Singular Links}
	\label{TOSDLinks}
\end{figure}
\end{example}

\newpage
\section*{Acknowledgement} 
Mohamed Elhamdadi was partially supported by Simons Foundation collaboration grant 712462.

\end{document}